\documentclass{article}
\usepackage{amssymb,amsmath}
\usepackage{graphicx}
\usepackage{amsfonts}
\usepackage{latexsym}
\usepackage{amsthm}
\usepackage{url}
\usepackage[iso]{isodateo}
\usepackage{capt-of}
\usepackage{hyperref}
\usepackage{epstopdf}
\usepackage{multirow}
\usepackage{tabularx}
\providecommand{\keywords}[1]{\textbf{\textit{keywords:}} #1}

\title{Existence and uniqueness of the maximum likelihood estimator for models with a Kronecker product covariance structure}
\author{Beata Ro\'s\textsuperscript{a}, Fetsje Bijma\textsuperscript{a}, Jan C. de Munck\textsuperscript{b} and Mathisca C.M. de Gunst\textsuperscript{a}}

\begin{document}
\maketitle
\noindent\textsuperscript{a} Department of Mathematics, Faculty of Exact Sciences, VU University Amsterdam, De Boelelaan 1081, 1081 HV, Amsterdam, The Netherlands\\
\textsuperscript{b} Department of Physics and Medical Technology, VU University Medical Center, De Boelelaan 1118,1081 HZ Amsterdam, The Netherlands\\

\begin{abstract}
This paper deals with multivariate Gaussian models for which the covariance matrix is a Kronecker product of two matrices. We consider maximum likelihood estimation of the model parameters, in particular of the covariance matrix. There is no explicit expression for the maximum likelihood estimator of a Kronecker product
covariance matrix. The main question in this paper is whether the maximum likelihood estimator of the covariance matrix exists and if it is unique. The answers are different for different models that we consider.\\\\
\end{abstract}

\section{Introduction}
\noindent
In many studies data are measured in multiple domains, like space, time and frequency. In this paper we focus on the situation where data are measured in two domains. An example of such data is multi-channel EEG, where several sensors on the scalp measure electric potential differences at sample rates of typically 200 to 2000 Hz \cite{DeMunck}. If a subject is repeatedly exposed to a  stimulus, the signal can be partitioned into separate measurements, 
one measurement being the signal in between two consecutive stimuli. If there are $p$ space points and $q$ time points, then each data measurement can be expressed as a $p$ by $q$ matrix $X$ (notation: $X\in\mathcal{M}_{p,q}\left(\mathbb{R}\right)$). Often it is a reasonable assumption that vec$\left(X\right)\sim\mathcal{N}\left(\mu,\Sigma\right)$ where the vec operator vectorizes a matrix by stacking its columns into one vector. In this paper we consider $n$ such data matrices, $X_{1},\ldots,X_{n}$, which for example represent the different measurements of EEG.

In certain situations, it can be assumed that the covariance matrix $\Sigma$ can be separated into two components. Each component represents the (co)variances within one domain. The assumption Cov$\left(X(i_{1},j_{1}),X(i_{2},j_{2})\right)=\Gamma(i_{1},i_{2})\Psi(j_{1},j_{2})$ for all $i_{1},i_{2},j_{1},j_{2}$ is equivalent to Cov$\left(\text{vec}\left(X\right)\right)=\Psi\otimes\Gamma$, where $\Gamma\in\mathcal{M}_{p,p}\left(\mathbb{R}\right)$, $\Psi\in\mathcal{M}_{q,q}\left(\mathbb{R}\right)$ and $\otimes$ denotes the Kronecker product.
In this paper it is assumed that vec$\left(X\right)\sim\mathcal{N}\left(\mu,\Psi\otimes\Gamma\right)$, and we consider the problem of estimating $\mu$, and  $\Psi\otimes\Gamma$ from a series of independent observations of $X$ by maximum likelihood.

Models with a Kronecker product covariance structure have been proposed for many different situations. For spatiotemporal estimation of the covariance structure of EEG/MEG data see, for instance, \cite{BijmaEtAl,DeMunck,Huizenga,Torresani}; other examples are spatio-temporal analysis of environmental data \cite{Pinel,Mardia}, missing data imputation for microarray or Netflix movie rating data  \cite{Netflix},
and multi-task learning for detecting land mines in multiple fields or recognizing faces between different subjects \cite{MultiTask}.
Several tests have been developed for testing whether or not the covariance matrix is separable, that is, is a Kronecker product of two matrices \cite{Zimmerman}. Two general references for properties of Kronecker products are \cite{VanLoan,Magnus}.

The focus of the present paper is on existence and uniqueness of the maximum likelihood estimator of the covariance matrix of a multivariate Gaussian model with Kronecker product covariance structure. Therefore, one aspect is whether the likelihood function attains a maximum for some positive definite $\widehat{\Psi\otimes\Gamma}$. If the answer is positive, the second question is whether there is a unique positive definite maximizer $\widehat{\Psi\otimes\Gamma}$ of the likelihood function.
Finding the maximum likelihood estimate of the mean is straightforward.
On the other hand, maximum likelihood estimates for the Kronecker product covariance matrix are typically obtained by numerical approximation, because no explicit solution of the likelihood equations for its components $\Gamma$ and
$\Psi$ exists.
An iterative method that alternates between updating the estimate of one component of the covariance matrix and updating the estimate of the other component has been proposed in \cite{Dutilleul}.
Convergence of this so-called flip-flop algorithm is studied in \cite{Dutilleul,lu_zim}.
In \cite{Werner} theoretical asymptotic properties of the flip-flop algorithm are considered and the algorithm's performance for small sample sizes is investigated with a simulation study.
The special case in which one of the two matrices adopts a persymmetric structure is treated in \cite{JanssonEtAl}, and an adaptation of the algorithm
for estimation of a  Kronecker product
of two Toeplitz matrices is discussed in \cite{Wirfalt}.
In this paper we do not investigate properties of the flip-flop algorithm, i.e.\ of the method used for finding an approximation of the maximum likelihood estimate. Instead, we study existence and uniqueness of the maximum likelihood estimator of a covariance matrix that has a Kronecker product structure.
The results given in this work show that existence and uniqueness of the maximum likelihood estimator
cannot be taken for granted.

Because $c\Psi\otimes\frac{1}{c}\Gamma=\Psi\otimes\Gamma$ for any $c>0$, the components $\Gamma$ and $\Psi$ of a Kronecker product are not identifiable from the Kronecker product. This issue of trivial non-uniqueness can be addressed by making an additional identifiability constraint. Because $\Gamma$ and $\Psi$ are positive definite, the constraint can be that we fix a particular diagonal element of $\Gamma$ or $\Psi$ to be equal to 1. Another possibility is to assume that the determinant of one of the two components is equal to 1. Using one of the above mentioned identifiability constraints does not restrict the model in any way. The (non-trivial) uniqueness of the Kronecker product estimator $\widehat{\Psi\otimes\Gamma}$ is therefore equivalent to the (non-trivial) uniqueness of the corresponding pair $\left(\hat{\Gamma},\hat{\Psi}\right)$ with an additional constraint of the above mentioned type. The reason is that there is a one-to-one correspondence between $\Psi\otimes\Gamma$ and $\left(\Gamma,\Psi\right)$ with such constraint. In this paper the term 'uniqueness' is solely used for uniqueness of $\widehat{\Psi\otimes\Gamma}$, or equivalently $\left(\hat{\Gamma},\hat{\Psi}\right)$ with an identifiability constraint, and not for uniqueness of $\hat{\Psi}$ and $\hat{\Gamma}$ separately.

Analysis of existence and uniqueness of the maximum likelihood estimator of the covariance matrix under the Kronecker product structure is not trivial, because the parameter space is not convex and the set of multivariate normal distributions with a Kronecker product covariance matrix is a curved exponential family.
These topics were also studied in the papers \cite{Dutilleul,von_Rosen} that have stimulated a lot of research in the area. A condition that is claimed to be necessary and sufficient for existence is given in \cite{Dutilleul} with inadequate justification. We cannot verify this condition, but instead prove existence under a stronger condition.
In \cite{Dutilleul} and \cite{von_Rosen} uniqueness is considered. We believe that the results about the conditions for uniqueness in those two studies are not correct and demonstrate this by counterexamples. In our paper we additionally consider models for which more restrictions are imposed on $\Gamma$ or $\Psi$. We give a novel proof for the existence and uniqueness of the maximum likelihood estimator
for the case that at least one of $\Gamma$ and $\Psi$ is constrained to be diagonal.
In the next section the model and the likelihood equations are presented,
and existence and uniqueness of the maximum likelihood estimator
for unrestricted models that are only constrained by positive definiteness are discussed.
In Section~3 our results for the model with both components diagonal are presented. In Section~4 the results are extended to the case with only one component diagonal.
We conclude with a discussion of our results.
Proofs are given in the Appendix.

\section{The general model}
\noindent
Suppose that the observations $X_{1},\ldots,X_{n}\in\mathcal{M}_{p,q}\left(\mathbb{R}\right)$ satisfy
\begin{equation}
\label{model}
\text{vec}\left(X_{k}\right)\sim\mathcal{N}\left(\mu,\Psi\otimes\Gamma\right),
\end{equation}
$k=1,\ldots,n$, and that $\text{vec}\left(X_{1}\right),\ldots,\text{vec}\left(X_{n}\right)$ are independent.
Here $\mu$ is a vector of length $pq$,
and $\Gamma\in\mathcal{M}_{p,p}\left(\mathbb{R}\right)$ and $\Psi\in\mathcal{M}_{q,q}\left(\mathbb{R}\right)$ are positive definite matrices; $\mu,\Gamma,\Psi$ are unknown parameters.
It is important to note that for $\Gamma,\Psi$ positive definite, the Kronecker product $\Psi\otimes\Gamma$ is also a positive definite matrix \cite{book_matrix}. This implies that $\Psi\otimes\Gamma$ is a covariance matrix of some normally distributed random vector. Therefore the model is well defined.

We consider maximum likelihood estimation for the vector $\mu$ and the covariance matrix $\Psi\otimes\Gamma$.  Let $M\in\mathcal{M}_{p,q}\left(\mathbb{R}\right)$ be such that vec$\left(M\right)=\mu$.
The likelihood function for $M, \Gamma$ and $\Psi$ is\\
$L(M,\Gamma,\Psi)=$
\begin{equation}
\label{likelihood}
\left(2\pi\right)^{-\frac{1}{2}pqn}
\left|\Gamma\right|^{-\frac{1}{2}qn}
\left|\Psi\right|^{-\frac{1}{2}pn}
\text{etr}\left(-\frac{1}{2}\Psi^{-1}\sum_{k=1}^{n}
{\left(X_{k}-M\right)^{T}\Gamma^{-1}\left(X_{k}-M\right)}\right),
\end{equation}
where $\text{etr}\left(A\right)=\exp\left(\text{trace}\left(A\right)\right)$,
and $|A|$ denotes the determinant of $A$.
If
\begin{displaymath}
\hat{M}=\frac{1}{n}\sum_{k=1}^{n}{X_{k}},
\end{displaymath}
then the maximum likelihood estimator for the mean vector $\mu$ is $\hat{\mu}=\text{vec}\left(\hat{M}\right)$. The likelihood equations for $\Gamma$ and $\Psi$ take the following form:

\begin{subequations}
\label{likelihood_equations}
\begin{align}
\Gamma&=\frac{1}{nq}\sum_{k=1}^{n}{\left(X_{k}-\hat{M}\right)
\Psi^{-1}\left(X_{k}-\hat{M}\right)^{T}},\label{likelihood_equations1}\\
\Psi&=\frac{1}{np}\sum_{k=1}^{n}{\left(X_{k}-\hat{M}\right)^{T}
\Gamma^{-1}\left(X_{k}-\hat{M}\right)}. \label{likelihood_equations2}
\end{align}
\end{subequations}
See, for instance  \cite{Bijma2004,Dutilleul} or \cite{Mardia} 
for a derivation of these equations.
Solving the first equation is equivalent to maximizing the likelihood with respect to $\Gamma$ for a fixed $\Psi$, and vice versa for the second equation \cite{Mardia}. Therefore, solutions of the set \eqref{likelihood_equations1} and \eqref{likelihood_equations2} correspond to local maxima of the likelihood function. In fact, one can study the likelihood equations to derive properties of the maximum likelihood estimator $\widehat{\Psi\otimes\Gamma}$ of the covariance matrix $\Psi\otimes\Gamma$. 
Indeed, because there is a one-to-one correspondence between the solutions of the likelihood equations and the set of all critical points of the likelihood function, all critical points of the likelihood function are local maxima.
It is worth noting that if the equations have a solution $\left(\Gamma,\Psi\right)$, the Kronecker product $\Psi\otimes\Gamma$ must be positive definite.
Moreover, if $\widehat{\Psi\otimes\Gamma}$ is the maximum likelihood estimator of $\Psi\otimes\Gamma$, it corresponds to a global (hence also a local) maximum of the likelihood function, and therefore the $\hat{\Gamma}$ and $\hat{\Psi}$ for which $\widehat{\Psi\otimes\Gamma}=\hat{\Psi}\otimes\hat{\Gamma}$ must satisfy the likelihood equations. If, on the other hand, no positive definite $\Gamma$ and $\Psi$ exist that satisfy \eqref{likelihood_equations1} and \eqref{likelihood_equations2}, the likelihood function does not have any local or global maxima on the set of positive definite Kronecker products. In such case the maximum likelihood estimator of the covariance matrix does not exist. Since there is this link between the solutions of the likelihood equations and the maximum likelihood estimator of $\Psi\otimes\Gamma$, one can study the likelihood equations to derive the properties of the maximum likelihood estimator of the covariance matrix. 

Due to the structure of the likelihood equations there are no explicit formulas for the solutions $\left(\Gamma,\Psi\right)$ .
We note that if \eqref{likelihood_equations1}
is satisfied for a pair $(\Gamma, \Psi)$, then the likelihood in $(\hat M, \Gamma, \Psi)$ attains the value
\begin{subequations}
\begin{align}
&\left(2\pi\right)^{-\frac{1}{2}pqn}
\left|\frac{1}{nq}\sum_{k=1}^{n}{\left(X_{k}-\hat{M}\right)\Psi^{-1}
\left(X_{k}-\hat{M}\right)^{T}}\right|^{-\frac{1}{2}qn}
\left|\Psi\right|^{-\frac{1}{2}pn}e^{-\frac{1}{2}npq}.\label{likelihood2a}\\
\intertext{If \eqref{likelihood_equations2} is satisfied for $\left(\Gamma,\Psi\right)$, the likelihood in $\left(\hat{M},\Gamma,\Psi\right)$ equals}
&\left(2\pi\right)^{-\frac{1}{2}pqn}\left|\frac{1}{np}\sum_{k=1}^{n}{\left(X_{k}-\hat{M}\right)^{T}\Gamma^{-1}\left(X_{k}-\hat{M}\right)}\right|^{-\frac{1}{2}pn}\left|\Gamma\right|^{-\frac{1}{2}qn}e^{-\frac{1}{2}npq}.\label{likelihood2b}
\end{align}
\end{subequations}

\subsection{Existence}
\noindent
It can happen that the maximum likelihood estimator of $\Psi\otimes\Gamma$ does not exist.
A necessary condition for the existence of the maximum likelihood estimator of the covariance matrix under model \eqref{model} has been derived in \cite{Dutilleul}
and is stated in the following theorem.
\newtheorem{necessary}{Theorem}
\begin{necessary}
\label{necessary}
Suppose $X_{1},\ldots,X_{n}\in\mathcal{M}_{p,q}\left(\mathbb{R}\right)$ satisfy model \eqref{model}. If the maximum likelihood estimator of $\Psi\otimes\Gamma$ exists, it must be that $n>\max\{\frac{p}{q},\frac{q}{p}\}$.
\end{necessary}

\noindent
In \cite{Dutilleul} it is claimed that this condition is also sufficient for the existence as well as for the uniqueness of the maximum likelihood estimator of $\Psi\otimes\Gamma$. However, uniqueness does not follow from this condition, as is shown in the next section. Moreover, it is not known whether it guarantees existence, because it is not sufficient to show that all updates of the flip-flop algorithm have full rank as is done in \cite{Dutilleul}. It could still happen that the sequence of updates converges (after infinitely many steps) to a Kronecker product that does not have a full rank with the likelihood converging to its supremum. The reason is that the space $\{\Psi\otimes\Gamma:\Gamma\in\mathcal{M}_{p,p}\left(\mathbb{R}\right),
\Psi\in\mathcal{M}_{q,q}\left(\mathbb{R}\right); \Gamma,\Psi \ \text{positive definite}\}$ with any norm is not closed.
Below we shall prove existence under a stronger condition on the sample size $n$. For this the following lemma will be needed.
\newtheorem{lemma}[necessary]{Lemma}
\begin{lemma}
\label{comp_KP}
Let $\mathcal{K}$ be the space of non-negative definite $pq\times pq$ matrices that have a Kronecker product structure such that each $K\in\mathcal{K}$ can be expressed as $K=A\otimes B$, where $A\in\mathcal{M}_{q,q}\left(\mathbb{R}\right)$, $B\in\mathcal{M}_{p,p}\left(\mathbb{R}\right)$ are also non-negative definite.
Let $\mathcal{K}$ be equipped with the Frobenius norm.
Then $\mathcal{K}$ is closed.
\end{lemma}

\begin{proof}
Let $K_{n}=A_{n}\otimes B_{n}\in\mathcal{K}$ and $K_{n}\rightarrow K$ as $n\rightarrow\infty$. If all diagonal elements of $K$ are zero, then $K$ must consist of only zeros. Therefore, in particular $K\in\mathcal{K}$. To consider the non-trivial case, assume that some of the diagonal elements of $K$ are not equal to zero. Without loss of generality, $K\left(1,1\right)\neq 0$.
Because $cA_n\otimes \frac{1}{c}B_n=A_n\otimes B_n$, we may assume that $A_{n}\left(1,1\right)=1$ for $n$ sufficiently large. By the definition of the Kronecker product it follows easily that
$B_{n}\rightarrow B$ as $n\rightarrow\infty$, where $B$ is equal to the upperleft $p\times p$ block of $K$.
But since $A_{n}\left(i,j\right)B_{n}$ converges to the $(i,j)$-th $p\times p$ block of $K$
as $n\rightarrow\infty$, this block should
be equal to $a_{ij}B$ for some value $a_{ij}$, $i,j=1,\ldots, q$.
If $A\in\mathcal{M}_{q,q}\left(\mathbb{R}\right)$ is
defined by $A\left(i,j\right)=a_{ij}$, we see that $K=A\otimes B$.
Moreover, by continuity arguments $A$, $B$ and $K$
are non-negative definite.
Hence, $K\in\mathcal{K}$ and therefore $\mathcal{K}$ is closed.
\end{proof}

\noindent
We now formulate the main result of this subsection.
\newtheorem{existence_new}[necessary]{Theorem}
\begin{existence_new}
\label{existence_new}
Let $X_{1},\ldots,X_{n}\in\mathcal{M}_{p,q}\left(\mathbb{R}\right)$ satisfy model \eqref{model}.
If $n> pq$, then the maximum likelihood estimator of $\Psi\otimes\Gamma$ exists with probability $1$.
\end{existence_new}

\begin{proof}
The proof uses a result from \cite{Burg}, for which the sample covariance matrix has to be positive definite. Therefore the assumption that $n>pq$ is made, which guarantees positive definiteness of the sample covariance matrix with probability $1$.
Let $R$ denote the covariance matrix $\Psi\otimes\Gamma$, and
 $S$ be the sample covariance matrix.
Furthermore, let $\mathcal{K}$ be defined as in Lemma~\ref{comp_KP}.
Then $R\in\mathcal{K}$.
By taking the logarithm of the likelihood function and dropping the terms that do not involve the covariance matrix $R$, we see that maximization of the likelihood
with respect to $\Gamma$ and $\Psi$ for $\Gamma$ and $\Psi$ positive definite matrices, is equivalent to maximization over $\mathcal{K}$ of the function $g(R)$ given by

\begin{equation}
\label{gR}
g\left(R\right)=-\text{log}\left|R\right|-\text{tr}\left(R^{-1}S\right).
\end{equation}

\noindent
Because $\mathcal{K}$ contains at least one positive definite matrix, the theorem now is an immediate consequence of Lemma~\ref{comp_KP}
and the result proved in \cite{Burg}
that if $S$ is positive definite and $R$ belongs to a closed subset $\mathcal{K}$
of the set of non-negative definite matrices, then there exists a
positive definite solution of the maximization of \eqref{gR} over $R\in \mathcal{K}$.
We note that the fact that here $\mu$ is unknown, whereas in \cite{Burg} $\mu$ is assumed to be $0$, only means that we need $n>pq$ instead of $n\geq pq$. Otherwise it has no consequences.
\end{proof}

\noindent
In conclusion, if $n\le \text{max}\{\frac{p}{q},\frac{q}{p}\}$, the maximum likelihood estimator of the covariance matrix does not exist. However, if $n>pq$, it exists with probability $1$.
As yet, there do not seem to be existence results for the case $n\in\left[\text{max}\{\frac{p}{q},\frac{q}{p}\}+1,pq\right]$.

\subsection{Uniqueness}
\label{nonuniqueness}
\noindent
In the papers \cite{Dutilleul} and  \cite{von_Rosen} the question about uniqueness of the maximum likelihood estimator of $\Psi\otimes\Gamma$ has been studied by investigating whether the likelihood equations have a unique solution. 
Unfortunately, the sufficient conditions derived in both papers do not hold in general, which we explain below.

In \cite{Dutilleul} and \cite{Lee} it is claimed that if $n>\max\{\frac{p}{q},\frac{q}{p}\}$, then there is a unique $\widehat{\Psi\otimes\Gamma}$ such that the corresponding pair $\left(\hat{\Gamma},\hat{\Psi}\right)$ satisfies the likelihood equations \eqref{likelihood_equations}.
The argument given is based on the results in Section 18.24 of \cite{Kendall}. However, these results are not applicable, because they only apply to random vectors having a distribution within an exponential family, whereas the $\mathcal{N}\left(\mu,\Psi\otimes\Gamma\right)$ distributions form
a curved exponential family rather than an exponential family. The definition and properties of an exponential and a curved exponential family can for instance be found in \cite{Keener}.

In \cite{von_Rosen} uniqueness is claimed if $n>\max\{p,q\}$. However, we found a counterexample, which is described in Section \ref{counterexample}. Therefore, part of the proof of Theorem~3.1 in that paper does not hold. This seems to be the part in which, using the notation of \cite{von_Rosen}, it is assumed that $P_{\Sigma_{1}}-P_{\Sigma_{2}}=0$ holds if and only if $\Sigma_{1}=\Sigma_{2}$. This is not true, though, because for all solutions of the likelihood equations in our counterexample,
the $\Sigma$ is different, but all $P_{\Sigma}$ are the same.

In \cite{lu_zim} results of a computational experiment were reported which showed that for $n=2$ or $n=3$, the flip-flop algorithm converged to many different estimates, depending on the starting value, while the likelihood function at each of these estimates was identical. Our results give a mathematical explanation of these empirical findings and confirm that the likelihood may have multiple points at which its global maximum is attained. Each of these points corresponds to a different solution of the likelihood equations.

\subsection{Non-uniqueness for $n=2$, $p=q$}
\noindent
In this section a counterexample to the statement that $n>\max\{\frac{p}{q},\frac{q}{p}\}$ would be sufficient for uniqueness of the maximum likelihood estimator of $\Psi\otimes\Gamma$ is given for $n=2$ and $p=q$. We will show that with probability $1$ there exist multiple solutions of the likelihood equations and each of them corresponds to a different maximum likelihood estimator of $\Psi\otimes\Gamma$.
Because for $n=2$, $\hat{M}=\frac{1}{2}\left(X_{1}+X_{2}\right)$, and $X_{1}-\hat{M}=-\left(X_{2}-\hat{M}\right)=\frac{1}{2}X_{1}-\frac{1}{2}X_{2}$, the likelihood equations for $\Gamma$ and $\Psi$ become
\begin{subequations}
\begin{align}
\Gamma&=\frac{1}{p}\left(\frac{1}{2}X_{1}-\frac{1}{2}X_{2}\right)
\Psi^{-1}\left(\frac{1}{2}X_{1}-\frac{1}{2}X_{2}\right)^{T},\label{first_2}\\
\Psi&=\frac{1}{p}\left(\frac{1}{2}X_{1}-\frac{1}{2}X_{2}\right)^{T}
\Gamma^{-1}\left(\frac{1}{2}X_{1}-\frac{1}{2}X_{2}\right).\label{second_2}
\end{align}
\end{subequations}

\noindent
Note that $X_{1}-X_{2}$ is a square matrix and $X_{1}-X_{2}$ is invertible with probability $1$. Let us then assume that $X_{1}-X_{2}$ is invertible. We obtain from  \eqref{first_2} that
\begin{displaymath}
p\left(\frac{1}{2}X_{1}-\frac{1}{2}X_{2}\right)^{-1}\Gamma\left(\left(\frac{1}{2}
X_{1}-\frac{1}{2}X_{2}\right)^{T}\right)^{-1}=\Psi^{-1},
\end{displaymath}
which is equivalent to  \eqref{second_2}.
This shows that in the case where $n=2$, $p=q$ and rank$\left(X_{1}-X_{2}\right)=p$, equation \eqref{first_2} implies equation \eqref{second_2}. Therefore one can take for $\Psi$ in \eqref{first_2} any positive definite $q\times q$ matrix $\hat{\Psi}$ and calculate the corresponding $\hat{\Gamma}$ from \eqref{first_2}.
Then $\hat{\mu}=\frac{1}{2}\sum_{k=1}^{2}{\text{vec}\left(X_{k}\right)}$ and $\hat{\Gamma}, \hat{\Psi}$ satisfy the likelihood equations.

Moreover, because for all different choices of $\hat{\Psi}$, and $\hat{\Gamma}$ satisfying \eqref{first_2} with this $\hat{\Psi}$, the likelihood function attains the value given in \eqref{likelihood2a},
and since for the case $n=2$, $p=q$,
\begin{displaymath}
\left|\hat{\Gamma}\right|^{-\frac{1}{2}qn}\left|\hat{\Psi}\right|^{-\frac{1}{2}pn}
=\left|\frac{1}{2}X_{1}-\frac{1}{2}X_{2}\right|^{-2p},
\end{displaymath}
we have that for any pair $\left(\hat{\Gamma},\hat{\Psi}\right)$ that satisfies the likelihood equations, the likelihood function equals
\begin{displaymath}
\left(2\pi\right)^{-p^{2}}p^{-p}
\left|\frac{1}{2}X_{1}-\frac{1}{2}X_{2}\right|^{-2p}\exp\left(-p^{2}\right),
\end{displaymath}
which does not depend on $\left(\hat{\Gamma},\hat{\Psi}\right)$.
This means that different solutions of the likelihood equations correspond to
different global maximizers of the likelihood function which all have the same likelihood value. We have actually shown that in this case the likelihood equations can be solved analytically and the space of maximum likelihood estimators of the covariance matrix is $\{\Psi\otimes\Gamma:\Psi\text{- positive definite},\Gamma=\frac{1}{p}\left(\frac{1}{2}X_{1}-\frac{1}{2}X_{2}\right)\Psi^{-1}\left(\frac{1}{2}X_{1}-\frac{1}{2}X_{2}\right)^{T}\}$.
Hence, if $n=2$, $p=q$, and $X_1-X_2$ is invertible,
 the maximum likelihood estimator of the covariance matrix $\Psi\otimes\Gamma$ is not unique.


\subsection{Non-uniqueness for $n=3$, $p=q=2$}
\label{counterexample}
\noindent
In Theorem 3.1 in \cite{von_Rosen} it is stated that if $n>\max\{p,q\}$, then the likelihood equations have a unique solution. In this section it is shown that for $n=3$ and $p=q=2$, uniqueness is not always the case. It strongly depends on the data, and non-uniqueness of $\left(\hat{\Gamma},\hat{\Psi}\right)$ that satisfies the likelihood equations and an identifiability constraint implies non-uniqueness of the maximum likelihood estimator $\widehat{\Psi\otimes\Gamma}$ of $\Psi\otimes\Gamma$.
Derivations of the results in this section can be found in \ref{nonuniqueness_von_Rosen}.
Moreover, it is shown in \ref{nonuniqueness_von_Rosen} that
for a given data set it can be determined
 whether there is unique maximum likelihood estimator for $\Psi\otimes\Gamma$  by computing the discriminant of a quadratic polynomial with coefficients that are
functions of the data.

If for some matrix $\Psi$, \eqref{likelihood_equations2}
is satisfied, then $\Psi=\Psi\left(\Gamma\right)$ and
the likelihood function can be written as a function of $\Gamma$ only.
The resulting profile likelihood needs to be maximized
with respect to $\Gamma$.
Let $\Gamma$ be parameterized by
\begin{displaymath}
\Gamma=
\begin{pmatrix}
a & b\\
b & 1\\
\end{pmatrix},
\end{displaymath}
with
\begin{equation}
\label{assumpab}
a>b^{2},
\end{equation}
and where
the second diagonal element is taken to be 1 in order to avoid identifiability problems.
Then the likelihood function can be considered as a function $L\left(a,b\right)$
of $a$ and $b$.
If $L$ attains a global maximum for
 some pair of values $(a,b)$, then
$a$ and $b$ need to satisfy
\begin{subequations}
\begin{align}
\frac{\partial}{\partial a}L\left(a,b\right)&=0\label{dLda},\\
\frac{\partial}{\partial b}L\left(a,b\right)&=0.\label{dLdb}
\end{align}
\end{subequations}
It turns out that each solution of \eqref{dLda} that also satisfies assumption
 \eqref{assumpab} can be expressed as
 $\left(a,b\right)=\left(g\left(b\right),b\right)
 =\left(g\left(b,X_{1},\ldots,X_{n}\right),b\right)$
 with $g\left(b\right)> b^{2}$.
Furthermore, under \eqref{assumpab}
solutions of  \eqref{dLdb} have two possible forms: $\left(a,b\right)=\left(h_{1}\left(b\right),b\right)
=\left(h_{1}\left(b,X_{1},\ldots,X_{n}\right),b\right)$
or $\left(a,b\right)=\left(h_{2}\left(b\right),b\right)
=\left(h_{2}\left(b,X_{1},\ldots,X_{n}\right),b\right)$.

Explicit expressions for $g,h_{1},h_{2}$ are derived in \ref{nonuniqueness_von_Rosen}. These formulas can be easily verified by analytical software such as Mathematica.
Hence, if $n=3$, $p=q=2$, the likelihood equations have analytical solutions. It can be shown that with positive probability $g=h_{1}$ on some open interval $I$. For any $b\in I$ we have that $g\left(b\right)>b^{2}$, so all $\left(g\left(b\right),b\right)$ for $b\in I$ are solutions of the set of equations \eqref{dLda} and \eqref{dLdb}. The function $L$ is constant on this set and it takes
the same global maximum in each $\left(g\left(b\right),b\right)$.
This shows that uniqueness of the maximum likelihood estimator does
not hold for $n=3, p=q=2$.
We note that for $n=3$ and for any choice of positive definite $2$ by $2$ matrices $\Gamma$ and $\Psi$ in model \eqref{model}
the probability of observing non-uniqueness is strictly positive.

To illustrate the above, two data sets were simulated from model \eqref{model}
with $n=3$, $\mu=0$ and
\begin{displaymath}
\Gamma=
\begin{pmatrix}
0.15 & 0.24\\
0.24 & 1\\
\end{pmatrix}
\end{displaymath}
and
\begin{displaymath}
\Psi=
\begin{pmatrix}
1.69 & 0.26\\
0.26 & 0.15\\
\end{pmatrix}.
\end{displaymath}

\begin{figure}
\begin{center}
\includegraphics[height=60mm]{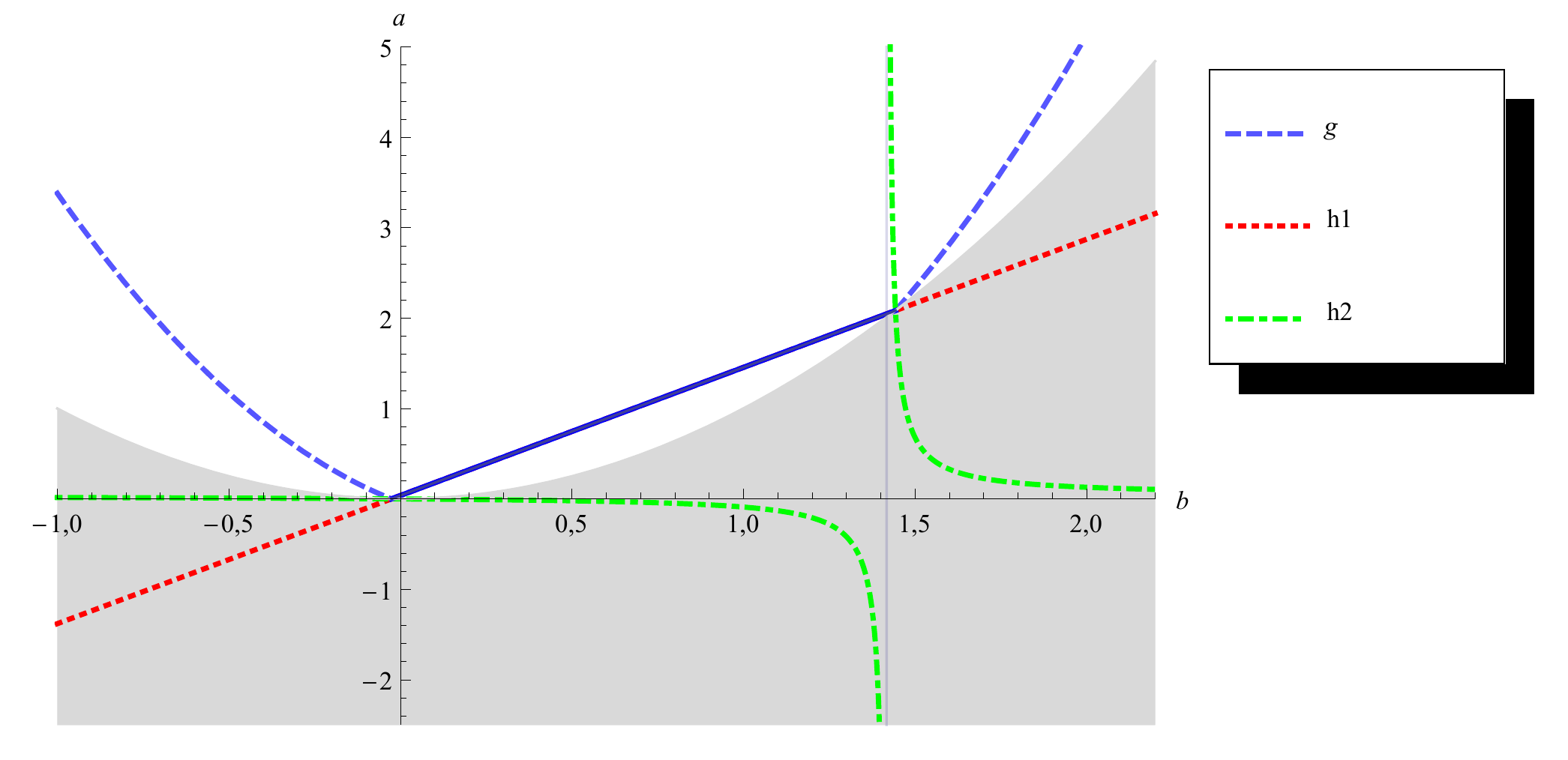}
\caption{Functions $g,h_{1},h_{2}$ for a data set
for which uniqueness does not hold in the case $n=3,p=q=2$.
Each point of the line segment corresponds to a different maximum likelihood estimator. The white area corresponds to $a>b^2$. }
\label{plot_nonuniqueness}
\includegraphics[height=60mm]{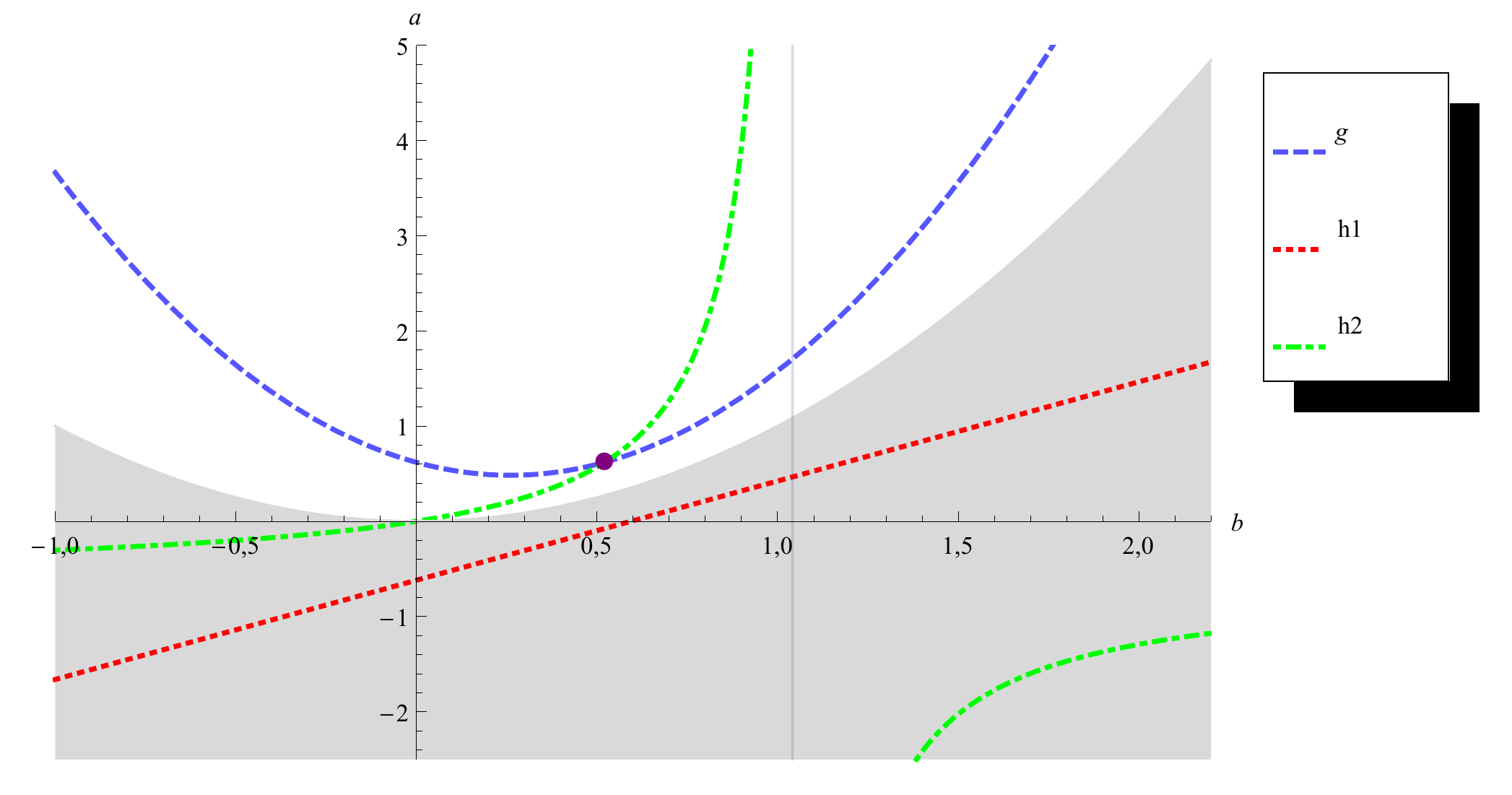}
\caption{Functions $g,h_{1},h_{2}$ for a data set for which uniqueness does hold
in the case $n=3,p=q=2$.
The point where $g$ and $h_{2}$ intersect corresponds to a unique maximum likelihood estimator.
The white area corresponds to  $a> b^2$.}
\label{plot_uniqueness}
\end{center}
\end{figure}

\noindent
For the two data sets the functions $g$, $h_{1}$, and $h_{2}$ are plotted in
Figures~ \ref{plot_nonuniqueness} and \ref{plot_uniqueness}. For the first data set uniqueness does not hold, whereas for the second one it does hold.
 This illustrates that uniqueness/non-uniqueness is data dependent.
A simulation study showed that for the above $\Gamma$ and $\Psi$ the probability of obtaining data for which non-uniqueness holds is approximately $0.8$.

\section{The diagonal model}
\label{sec3}
\noindent
In this section we consider existence and uniqueness
of the maximum likelihood estimator for $\Psi\otimes\Gamma$ of the diagonal model,
this is model \eqref{model} with the additional assumption that both matrices
$\Psi$ and $\Gamma$ are diagonal.
Let $\Gamma=diag\left(\gamma_{1},\ldots,\gamma_{p}\right)$, and $\Psi=diag\left(\psi_{1},\ldots,\psi_{q}\right)$ and $\gamma_{i},\psi_{j}>0$.
Note that $\Psi\otimes\Gamma$ is a diagonal matrix with elements $\gamma_{i}\psi_{j}$ on the diagonal. Because $\Psi\otimes\Gamma$ is a diagonal covariance matrix of a normally distributed random vector vec$\left(X_{k}\right)$, this model assumes that each $X_{k}$ consists of independent elements.

It will be convenient to introduce some additional notation. For $i=1,\ldots,p$, $j=1,\ldots,q$, let $Y^{2}_{i,j}$ be defined by
\begin{displaymath}
Y^{2}_{i,j}=\sum_{k=1}^{n}{\left(X_{k}(i,j)-\hat{M}(i,j)\right)^{2}}.
\end{displaymath}
Note that for $n\geq 2$, $P(Y_{i,j}^{2}>0)=1$, because the $X_{k}(i,j)$ are continuous
random variables.
The likelihood function for the parameters in terms of the $Y_{i,j}^{2}$ is
\begin{equation}
\label{likdiag}
L(\gamma_{1},\ldots,\gamma_{p},\psi_{1},\ldots,\psi_{q})
=\left(2\pi\right)^{-\frac{pqn}{2}}\prod_{i,j}{\left(\left(\gamma_{i}
\psi_{j}\right)^{-\frac{n}{2}}\exp\left(-\frac{1}{2}\frac{Y_{i,j}^{2}}
{\gamma_{i}\psi_{j}}\right)\right)},
\end{equation}
and the likelihood equations are
\begin{subequations}
\label{likeqdiag}
\begin{align}
\gamma_{i} &=\frac{1}{nq}\sum_{j=1}^{q}{\frac{1}{\psi_{j}}Y_{i,j}^{2}}, \qquad i=1,\ldots,p,\label{eq_double_diag1}\\
\psi_{j} &=\frac{1}{np}\sum_{i=1}^{p}{\frac{1}{\gamma_{i}}Y_{i,j}^{2}}, \qquad  j=1,\ldots,q. \label{eq_double_diag2}
\end{align}
\end{subequations}
In the remainder of the section it is assumed that $n\geq 2$.

\subsection{Existence}\label{double_diag_existence}
\noindent
In this section it will be proved that with probability 1 there exists the maximum likelihood estimator $\widehat{\Psi\otimes\Gamma}$ of the covariance matrix $\Psi\otimes\Gamma$ with diagonal $\Gamma$ and $\Psi$.
\newtheorem{exist_diag}[necessary]{Theorem}
\begin{exist_diag}
\label{exist_diag}
Let $X_{1},\ldots,X_{n}\in\mathcal{M}_{p,q}\left(\mathbb{R}\right)$ satisfy model \eqref{model} with $n\geq 2$ and the additional assumption that $\Gamma$ and $\Psi$ are diagonal matrices. Then the maximum likelihood estimator of $\Psi\otimes\Gamma$ exists with probability $1$.
\end{exist_diag}
\begin{proof}
Because for a given value of $\Gamma$, the likelihood function is maximized for $\Psi$ defined by \eqref{eq_double_diag2}, this value of $\Psi$ can be inserted in $L$, which, as a result, only depends on $\Gamma$. The convenient choice of the identifiability constraint is now $\left|\Gamma\right|=1$. The likelihood then becomes
\begin{align*}
L\left(\Gamma\right|X_{1},\ldots,X_{n})&=\left(2\pi\right)^{-\frac{pqn}{2}}\left(\prod_{j=1}^{q}{\left(\frac{1}{np}\sum_{i=1}^{p}{\frac{1}{\gamma_{i}}Y_{i,j}^{2}}\right)}\right)^{-\frac{1}{2}pn}\exp\left(-\frac{1}{2}pq\right).
\end{align*}
It is sufficient to show that this likelihood function attains its maximum for some positive definite $\hat{\Gamma}$. Such $\hat{\Gamma}$ can be used to obtain $\hat{\Psi}$ from \eqref{eq_double_diag2}. This $\hat{\Psi}$ will be positive definite with probability $1$, because $Y^{2}_{i,j}>0$ for $i=1,\ldots,p$, $j=1,\ldots,q$ with probability $1$.

First we have that maximizing $L$ with respect to $\Gamma$ is equivalent to minimizing $\left(\prod_{j=1}^{q}{\left(\frac{1}{np}\sum_{i=1}^{p}{\frac{1}{\gamma_{i}}Y_{i,j}^{2}}\right)}\right)$ with respect to $\left(\gamma_{1},\ldots,\gamma_{p}\right)$. This factor can be expressed as $\left|\sum_{i=1}^{p}{\frac{1}{\gamma_{i}}S_{i}}\right|$, where $S_{i}$ is a $q\times q$ diagonal matrix with the $j$th diagonal term equal to $\frac{1}{np}Y_{i,j}^{2}$, $j=1,\ldots,q$.
Hence we obtain the following constrained minimization problem
\begin{equation*}
\begin{aligned}
&\underset{\gamma_{1},\ldots,\gamma_{p}>0}{\text{argmin}}\left|\sum_{i=1}^{p}{\frac{1}{\gamma_{i}}S_{i}}\right|\\
&\text{ subject to } \prod_{i=1}^{p}{\gamma_{i}}=1.
\end{aligned}
\end{equation*}
From the Minkowski determinant theorem it follows that
\begin{align}
\label{S_inequality}
\left|\sum_{i=1}^{p}{\frac{1}{\gamma_{i}}S_{i}}\right|\geq\sum_{i=1}^{p}{\left|\frac{1}{\gamma_{i}}S_{i}\right|}.
\end{align}

The next step is to show that if some of $\gamma_{1},\ldots,\gamma_{p}$ are close to zero or $+\infty$, the objective function attains high values.
Let $L_{i}=\left(\frac{\left|S_{i}\right|}{\left|S_{1}+\ldots +S_{p}\right|}\right)^{\frac{1}{q}}$ and $L=\min\{L_{1},\ldots,L_{p}\}$, then, from \eqref{S_inequality}, $L\leq 1$.
We see from the definition of $L$ that if some $\gamma_{i}<L$, $\left|\frac{1}{\gamma_{i}}S_{i} \right|=\left(\frac{1}{\gamma_{i}}\right)^{q}\left|S_{i}\right|>\left|S_{1}+\ldots+S_{p}\right|$, and therefore
\begin{align*}
\left|\sum_{i=1}^{p}{\frac{1}{\gamma_{i}}S_{i}}\right|>\left|\sum_{i=1}^{p}{S_{i}}\right|.
\end{align*}
The right side of this inequality equals the value of the objective function for $\left(\gamma_{1},\ldots,\gamma_{p}\right)=\left(1,\ldots,1\right)$.
Therefore minimizing $\left|\sum_{i=1}^{p}{\frac{1}{\gamma_{}i}S_{i}}\right|$ on $\{\left(\gamma_{1},\ldots,\gamma_{p}\right):\left(\gamma_{1},\ldots,\gamma_{p}\right)\in\mathbb{R}_{+}^{p}, \prod_{i}{\gamma_{i}}=1\}$ is equivalent to minimizing it on the set $\{\left(\gamma_{1},\ldots,\gamma_{p}\right):\left(\gamma_{1},\ldots,\gamma_{p}\right)\in\left[L,+\infty\right)^{p}, \prod_{i}{\gamma_{i}}=1\}$. If we now take into account that $\prod_{i}{\gamma_{i}}=1$, it turns out that $\gamma_{i}\leq\frac{1}{L^{p-1}}$.
As a result, the set $\{\left(\gamma_{1},\ldots,\gamma_{p}\right):\left(\gamma_{1},\ldots,\gamma_{p}\right)\in\left[L,+\infty\right)^{p}, \prod_{i}{\gamma_{i}}=1\}$ equals
\begin{align*}
\{\left(\gamma_{1},\ldots,\gamma_{p}\right):\left(\gamma_{1},\ldots,\gamma_{p}\right)\in\left[L,\frac{1}{L^{p-1}}\right]^{p}, \prod_{i}{\gamma_{i}}=1\}.
\end{align*}
Because this is a compact set, there exists $\left(\hat{\gamma}_{1},\ldots,\hat{\gamma}_{p}\right)$, which belongs to this set that minimizes the function $\left|\sum_{i=1}^{p}{\frac{1}{\gamma_{i}}S_{i}}\right|$. This $\left(\hat{\gamma}_{1},\ldots,\hat{\gamma}_{p}\right)$ and $\left(\hat{\psi}_{1},\ldots,\hat{\psi}_{q}\right)$ obtained by inserting $\hat{\gamma}_{i}$ for $\gamma_{i}$ in \eqref{eq_double_diag2} correspond to $\widehat{\Psi\otimes\Gamma}$ that maximizes the likelihood function. 
\end{proof}

\subsection{Uniqueness}
\noindent
We will show that under the diagonal model with probability 1 there is at most one solution
of the likelihood equations. Combining this with the existence result from Theorem \ref{exist_diag}, we will have that for $n\geq 2$ the maximum likelihood estimator $\widehat{\Psi\otimes\Gamma}$ under the diagonal model exists and is unique with probability $1$.
For proving  uniqueness we need the following result.
\newtheorem{matrixB}[necessary]{Theorem}
\begin{matrixB}
\label{matrixB}
Let $B\in\mathcal{M}_{p,q}\left(\mathbb{R}\right)$ be parameterized as follows: $B(i,j)=\gamma_{i}\psi_{j}-\lambda_{i}\phi_{j}$, for some $\gamma_{i},\psi_{j},\lambda_{i},\phi_{j}>0$, $i=1,\ldots,p$, $j=1,\ldots,q$.
If $B$ has the property that in each column and in each row
either all elements are zero or there is at least one positive
and one negative element, then all elements of $B$ are zero.
\end{matrixB}

\begin{proof}
See Appendix \ref{uniqueness_diag}.
\end{proof}

\noindent
The following theorem states the main result of this section.

\newtheorem{uniqueness}[necessary]{Theorem}
\begin{uniqueness}
\label{uniqueness}
Let $X_{1},\ldots,X_{n}\in\mathcal{M}_{p,q}\left(\mathbb{R}\right)$ satisfy model \eqref{model} with $n\geq 2$ and the additional assumption that $\Gamma$ and $\Psi$ are diagonal matrices. Then with probability $1$ there exists a unique maximum likelihood estimator of the covariance matrix $\Psi\otimes\Gamma$.
\end{uniqueness}
\begin{proof}
Existence of the maximum likelihood estimator was shown in Theorem \ref{exist_diag}. We will first show that any two solutions of the likelihood equations are equivalent.
Let $\left(\Gamma,\Psi\right)=(diag(\gamma_{1},\ldots,
\gamma_{p}),diag(\psi_{1},\ldots,\psi_{q}))$  with
 $\gamma_{i}>0$, $i\in\{1,\ldots,p\}$, and $\psi_j>0$, $j\in\{1,\ldots,q\}$,
 be a solution of the likelihood equations
\eqref{likeqdiag} for the diagonal model.
Assume that $\left(\Lambda,\Phi\right)$ with $\Lambda=diag\left(\lambda_{1},\ldots,\lambda_{p}\right)$ and $\Phi=diag\left(\phi_{1},\ldots,\phi_{q}\right)$, where
$\lambda_{i}>0$, $i\in\{1,\ldots,p\}$, and $\phi_j>0$, $j\in\{1,\ldots,q\}$,
is another solution of \eqref{likdiag}. We will show that $\Psi\otimes\Gamma=\Phi\otimes\Lambda$.
Consider the matrix $B$ with $B(i,j)=\widetilde{\gamma}_{i}\widetilde{\psi}_{j}-\widetilde{\lambda}_{i}
\widetilde{\phi}_{j}$.
where $\widetilde{\gamma}_{i}=\frac{1}{\gamma_{i}}$, $\widetilde{\psi}_{j}=\frac{1}{\psi_{j}}$, $\widetilde{\lambda}_{i}=\frac{1}{\lambda_{i}}$, $\widetilde{\phi}_{j}=\frac{1}{\phi_{j}}$.

\noindent
From \eqref{likeqdiag} it follows that
\begin{displaymath}
\sum_{i=1}^{p}{\left(\widetilde{\gamma}_{i}\widetilde{\psi_{j}}-
\widetilde{\lambda_{i}}\widetilde{\phi_{j}}\right)Y_{i,j}^{2}}=0, \qquad j=1,\ldots,q,
\end{displaymath}
\begin{displaymath}
\sum_{j=1}^{q}{\left(\widetilde{\gamma}_{i}\widetilde{\psi}_{j}-\widetilde{\lambda}_{i}
\widetilde{\phi}_{j}\right)Y_{i,j}^{2}}=0, \qquad i=1,\ldots,p.
\end{displaymath}

\noindent
Because $Y^{2}_{i,j}> 0$ with probability 1 for all $i,j$,
we have with probability 1 that either $\widetilde{\gamma}_{i}\widetilde{\psi}_{j}-\widetilde{\lambda}_{i}\widetilde{\phi}_{j}=0$,
for all $i$ and $j$,
or for each $i$ and for each $j$ there is at least one positive
and one negative $\widetilde{\gamma}_{i}\widetilde{\psi}_{j}
-\widetilde{\lambda}_{i}\widetilde{\phi}_{j}$.
But this means that the matrix $B$ satisfies the assumptions
of Theorem \ref{matrixB}.  Hence, it holds that with probability 1, $\frac{1}{\gamma_{i}\psi_{j}}=\frac{1}{\lambda_{i}\phi_{j}}$, for all $i$ and $j$.
This means that $\Psi\otimes\Gamma=\Phi\otimes\Lambda$.

Now we consider the maximum likelihood estimator of the covariance matrix $\Psi\otimes\Gamma$. From Theorem \ref{exist_diag}, we know that the likelihood function attains its global maximum. Because this global maximum is also a local maximum, the likelihood equations must be satisfied at this point. The fact that any two solutions of the likelihood equations correspond to the same Kronecker product implies that the maximum likelihood estimator of $\Psi\otimes\Gamma$ is unique.
\end{proof}

In some situations, the mean vector $\mu$ of the model  is known and does not need to
be estimated. We then have the following result.

\newtheorem{corr}[necessary]{Corollary}
\begin{corr}
\label{corr}
Let $X_{1},\ldots,X_{n}\in\mathcal{M}_{p,q}\left(\mathbb{R}\right)$ satisfy model \eqref{model} with known mean vector $\mu=\text{vec}\left(M\right)$ and the additional assumption that $\Gamma$ and $\Psi$ are diagonal matrices. If $n\geq 1$, then with probability $1$ there exists a unique maximum likelihood estimator of the covariance matrix $\Psi\otimes\Gamma$.
\end{corr}
\begin{proof}
The proof is similar to the proof for the case where $\mu$ is unknown, but
now with $Y^{2}_{i,j}=\sum_{k=1}^{n}X_{k}(i,j)^{2}$
instead of $Y^{2}_{i,j}=\sum_{k=1}^{n}{\left(X_{k}(i,j)-\hat{M}(i,j)\right)^{2}}$. Here $n\geq 1$ is sufficient, because we do not lose a degree of freedom for estimating the mean vector.
\end{proof}

\section{The model with one diagonal component}
\noindent
We now consider the case where only one out of two covariance components is diagonal, while the other does not have additional restrictions other than positive definiteness. Without loss of generality, it is assumed that the model satisfies \eqref{model} with the additional assumption that $\Gamma=\text{diag}\left(\gamma_{1},\ldots,\gamma_{p}\right)$ with $\gamma_{1},\ldots,\gamma_{p}>0$. Under these assumptions, the likelihood equations for the covariance components are
\begin{subequations}
\label{lik_eq2KP}
\begin{align}
\Gamma &=\frac{1}{nq}\sum_{k=1}^{n}{\text{diag}\left(\left(X_{k}-\hat{M}\right)\Psi^{-1}\left(X_{k}-\hat{M}\right)^{T}\right)} \label{eq_diag}\\
\Psi &=\frac{1}{np}\sum_{k=1}^{n}{\left(X_{k}-\hat{M}\right)^{T}\Gamma^{-1}\left(X_{k}-\hat{M}\right)} \label{eq_nondiag}
\end{align}
\end{subequations}
\subsection{Existence}
\noindent
It will be shown that the condition $n>q$ guarantees existence of the maximum likelihood estimator with probability $1$, which is expressed by the following Theorem.
\newtheorem{existence_diagonal}[necessary]{Theorem}
\begin{existence_diagonal}
\label{existence_diagonal}
Let $X_{1},\ldots,X_{n}\in\mathcal{M}_{p,q}\left(\mathbb{R}\right)$ satisfy model \eqref{model} with the additional assumption that $\Gamma$ is a diagonal matrix. If $n>q$, then the maximum likelihood estimator of $\Psi\otimes\Gamma$ exists with probability $1$.
\end{existence_diagonal}
\begin{proof}
Reasoning as in the proof of Theorem  \ref{exist_diag}, we define the constrained minimization problem
\begin{equation*}
\begin{aligned}
&\underset{\gamma_{1},\ldots,\gamma_{p}>0}{\text{argmin}}\left|\sum_{i=1}^{p}{\frac{1}{\gamma_{i}}S_{i}}\right|\\
&\text{ subject to } \prod_{i=1}^{p}{\gamma_{i}}=1,
\end{aligned}
\end{equation*}
where now $S_{i}$ is the $q\times q$ matrix with $j_{1},j_{2}$th element equal to \\$\sum_{k=1}^{n}{\left(X_{k}\left(i,j_{1}\right)-\hat{M}\left(i,j_{1}\right)\right)\left(X_{k}\left(i,j_{2}\right)-\hat{M}\left(i,j_{2}\right)\right)}$. This means that $S_{i}$ is proportional to the sample covariance matrix restricted to row $i$ and because $n>q$, $S_{1},\ldots,S_{p}$ are positive definite with probability 1. Repeating the arguments in the proof of Theorem \ref{exist_diag}, it can be concluded that there exists $\left(\hat{\gamma}_{1},\ldots,\hat{\gamma}_{p}\right)$ that minimizes $\left|\sum_{i=1}^{p}{\frac{1}{\gamma_{i}}S_{i}}\right|$. It can be used to obtain $\hat{\Psi}$ from \eqref{eq_nondiag} such that for $\hat{\Gamma}=\text{diag}\left(\hat{\gamma}_{1},\ldots,\hat{\gamma}_{p}\right)$, $\widehat{\Psi\otimes\Gamma}=\hat{\Psi}\otimes\hat{\Gamma}$ maximizes the likelihood function.
\end{proof}
\subsection{Uniqueness for the case $p=2$}
\noindent
For  $n>q$, and the additional assumption that the diagonal component is a 2 by 2 matrix ($p=2$), we now prove that the maximum likelihood estimator of the covariance matrix is unique.
\newtheorem{uniqueness_1d}[necessary]{Theorem}
\begin{uniqueness_1d}
\label{uniqueness_1d}
Let $X_{1},\ldots,X_{n}\in\mathcal{M}_{p,q}\left(\mathbb{R}\right)$ satisfy model \eqref{model} with the additional assumption that $\Gamma$ is a diagonal matrix. If $n>q$ and $p=2$, then the maximum likelihood estimator of $\Psi\otimes\Gamma$ is unique with probability $1$.
\end{uniqueness_1d}
\begin{proof}
The same as in the proof of Theorem \ref{existence_diagonal}, we express maximum likelihood estimation as the following constrained minimization problem 
\begin{equation}
\begin{aligned}
\label{constr_pro_p2}
\underset{\gamma_{1},\gamma_{2}>0}{\text{argmin}}
\left|\frac{1}{\gamma_{1}}S_{1}+\frac{1}{\gamma_{2}}S_{2}\right|\\
\text{subject to } \gamma_{1}\gamma_{2}=1.
\end{aligned}
\end{equation}
Because $S_{1},S_{2}$ are positive definite, there exists a simultaneous diagonalization (matrix $A$ such that $A^{T}S_{1}A=\widetilde{S}_{1}$ and $A^{T}S_{2}A=\widetilde{S}_{2}$ are diagonal). Therefore
\begin{displaymath}
 \underset{\gamma_{1},\gamma_{2}>0}{\text{argmin}}
\left|\frac{1}{\gamma_{1}}S_{1}+\frac{1}{\gamma_{2}}S_{2}\right|=
\underset{\gamma_{1},\gamma_{2}>0}{\text{argmin}}
\left|\frac{1}{\gamma_{1}}\widetilde{S}_{1}+\frac{1}{\gamma_{2}}\widetilde{S}_{2}\right|,
\end{displaymath}
 subject to $\gamma_{1}\gamma_{2}=1$. Because $\widetilde{S}_{1},\widetilde{S}_{2}$ are diagonal, this minimization problem is equivalent to the case of both components being diagonal. From the properties of the maximum likelihood estimator of the covariance matrix under the diagonal model, we have that the constrained minimization problem \eqref{constr_pro_p2} has a solution that is unique. This solution corresponds to the unique maximum likelihood estimator of $\Psi\otimes\Gamma$.
\end{proof}
\section{Discussion}
\noindent
For three multivariate normal models with a Kronecker product covariance
matrix---the unrestricted Kronecker product model, the diagonal
Kronecker product model and the model with one component diagonal and one unrestricted--- we studied maximum likelihood estimation for the mean vector as well as for the
covariance matrix. Because in practice Kronecker product
models are more and more used in cases where $n<<pq$, it is important
to not only consider large $n$, but to investigate what happens for
smaller $n$ as well.

The diagonal model has good properties with respect to existence and uniqueness of the maximum likelihood estimator of the covariance matrix. It holds that for $n\geq 2$ with probability $1$ there is a unique product $\widehat{\Psi\otimes\Gamma}$ that maximizes the likelihood function.

Contrary to the diagonal model the unrestricted model has not been completely
successfully examined yet with respect to the uniqueness and existence of the maximum likelihood estimator of the covariance matrix. Although \cite{Dutilleul} declares a particular condition
to be necessary and sufficient for the existence of the maximum likelihood estimator, in fact it is not known whether this condition guarantees existence. We proved that a stronger condition is sufficient, and showed that neither of the commonly suggested conditions ($n>\max\{\frac{p}{q},\frac{q}{p}\}$, $n>\max\{p,q\}$) guarantees uniqueness.
Our results are in line with numerical studies described in \cite{lu_zim}. The counterexamples suggest that while estimating $\Psi\otimes\Gamma$ for the model  vec$\left(X\right)\sim\mathcal{N}\left(\mu,\Psi\otimes\Gamma\right)$ by
maximum likelihood, one should be aware of possible problems with uniqueness or existence of the covariance matrix estimator. In practice this means that when the flip-flop algorithm or any other numerical procedure is used to obtain an approximation of the maximum likelihood estimate,  the resulting  value of the computational procedure is in the case of non-uniqueness an approximation to only one of the possible  maximum likelihood estimates, whereas in the case of non-existence the resulting value cannot even be (an approximation of) a maximum likelihood estimate.

Finally, the model with only one diagonal component inherits some of the properties from the diagonal model. It turns out that if the sample size is bigger than the dimension of the unrestricted component, there exists the maximum likelihood estimator of the covariance matrix. Moreover, we have shown that the estimator is unique if the diagonal component is a $2$ by $2$ matrix.

\appendix

\section{Maximizing $L\left(a,b\right)$}
\label{nonuniqueness_von_Rosen}
\noindent
Recall that
$\Gamma=
\begin{pmatrix}
a & b\\
b & 1\\
\end{pmatrix}$
and that we have condition \eqref{assumpab} which says that $a>b^{2}$.
Because we have assumed that \eqref{likelihood_equations2}
is satisfied, and $n=3, p=q=2$, it follows from \eqref{likelihood}
that the likelihood function satisfies
\begin{displaymath}
L\left(a,b\right)=\frac{e^{-6}}{\left(2\pi\right)^{6}\left|
\begin{pmatrix}
a & b\\
b & 1\\
\end{pmatrix}
\right|^{3}\left|\Psi\left(
\begin{pmatrix}
a & b\\
b & 1\\
\end{pmatrix}
\right)\right|^{3}},
\end{displaymath}
where $\Psi(\Gamma)=\frac{1}{6}\sum_{k=1}^{3}{\left(X_{k}-\hat{M}\right)^{T}
\Gamma^{-1}\left(X_{k}-\hat{M}\right)}$.
Obviously, maximization of  $L$ with respect to $a$ and $b$ is equivalent to maximization of
\begin{displaymath}
\widetilde{L}\left(a,b\right)=\left|
\begin{pmatrix}
a & b\\
b & 1\\
\end{pmatrix}
\right|^{-1}
\left|\Psi\left(
\begin{pmatrix}
a & b\\
b & 1\\
\end{pmatrix}
\right)\right|^{-1}
\end{displaymath}
with respect to $a$ and $b$.
Because
$\begin{pmatrix}
a & b\\
b & 1\\
\end{pmatrix}
^{-1}=
\frac{1}{a-b^{2}}
\begin{pmatrix}
1 & -b\\
-b & a\\
\end{pmatrix}$
\\
and from the properties of the determinant, straightforward calculation yields
\begin{displaymath}
\widetilde{L}\left(a,b\right)=\frac{a-b^{2}}{\left|\frac{1}{6}
\sum_{k=1}^{3}{\left(X_{k}-\hat{M}\right)^{T}
\begin{pmatrix}
1 & -b\\
-b & a\\
\end{pmatrix}
\left(X_{k}-\hat{M}\right)}\right|}.
\end{displaymath}
which can be written as
\begin{displaymath}
\widetilde{L}\left(a,b\right)=\frac{a-b^{2}}{C+B_{1}b+A_{1}a+ABab+B_{2}b^{2}+A_{2}a^{2}}
\end{displaymath}
for some constants $A_{1}, B_{1}, AB, A_{2}, B_{2}$ and $C$ that only depend on the data.
Solving the likelihood equations \eqref{dLda} and \eqref{dLdb} for $a$
and $b$ under the assumption
\eqref{assumpab} thus amounts to
solving
\begin{equation}
\label{dLtda}
\frac{\partial}{\partial a}\widetilde{L}\left(a,b\right)=0
\end{equation}
and
\begin{equation}
\label{dLtdb}
\frac{\partial}{\partial b}\widetilde{L}\left(a,b\right)=0
\end{equation}
for $a$ and $b$ under \eqref{assumpab}.
Solving \eqref{dLtda} for $a$ with \eqref{assumpab}, we obtain one solution $a=g\left(b\right)$.
Solving \eqref{dLtdb} with respect to $a$, yields two solutions,
$a=h_{1}\left(b\right)$ and
$a=h_{2}\left(b\right)$.
The functions $g(b)$, $h_{1}(b)$ and $h_{2}(b)$ are defined by
\begin{equation}
\label{ghh}
\begin{array}{rl}
g\left(b\right) = & b^{2}+\left|b^2+V_1b+V_2\right|,\\
h_{1}\left(b\right)  = & -V_1b-V_2,\\
h_{2}\left(b\right)= & -V_2b/(b+V_3),
\end{array}
\end{equation}
where $V_1$ and $V_2$ are
functions of the data, this is, of the matrix elements of $X_{1},X_{2},X_{3}$.
They are given by
\begin{align*}
V_1=\frac{-r_{2,2} r_{3,1}+r_{2,1} r_{3,2}+r_{1,2} r_{4,1}-r_{1,1} r_{4,2}}{-r_{3,2} r_{4,1}+r_{3,1} r_{4,2}}, &
\quad
V_2=\frac{-r_{1,2} r_{2,1}+r_{1,1} r_{2,2}}{-r_{3,2} r_{4,1}+r_{3,1} r_{4,2}},&\\
V_3=\frac{-r_{2,2} r_{3,1}+r_{2,1} r_{3,2}+r_{1,2} r_{4,1}-r_{1,1} r_{4,2}}{-r_{3,2} r_{4,1}+r_{3,1} r_{4,2}},
\end{align*}
where for $k=1,2,3$,
\begin{displaymath}
r_{1,k}=X_{k}(1,1)-\frac{1}{3}\sum_{s=1}^{3}{X_{s}(1,1)},
\qquad
r_{2,k}=X_{k}(1,2)-\frac{1}{3}\sum_{s=1}^{3}{X_{s}(1,2)},
\end{displaymath}
\begin{displaymath}
r_{3,k}=X_{k}(2,1)-\frac{1}{3}\sum_{s=1}^{3}{X_{s}(2,1)},
\qquad
r_{4,k}=X_{k}(2,2)-\frac{1}{3}\sum_{s=1}^{3}{X_{s}(2,2)}.
\end{displaymath}
We note that all points $\left(a,b\right)$ such that $a=g\left(b\right)=h_{1}\left(b\right)$ or $a=g\left(b\right)=h_{2}\left(b\right)$ are solutions of both equations
\eqref{dLtda} and \eqref{dLtdb}. If additionally, \eqref{assumpab} holds, these solutions
 result in positive definite maximum likelihood
estimators of $\Gamma$ and $\Psi$.
It turns out that in order to investigate whether or not
the maximum likelihood estimator is
unique, one only needs to check the discriminant of the quadratic polynomial $W\left(b\right)$
defined by
\begin{displaymath}
W\left(b\right)=b^2+V_1 b+V_2.
\end{displaymath}
To see this, we first note that $W$ is a second degree polynomial in $b$
with coefficients that depend only on the data.
Next, since  from \eqref{ghh} we have that $g\left(b\right)=b^{2}+\left|W\left(b\right)\right|$
and $h_{1}\left(b\right)=b^{2}-W\left(b\right)$, it follows that
 \begin{displaymath}
   g\left(b\right) = \left\{
     \begin{array}{lr}
       h_{1}\left(b\right), &  W\left(b\right)<0,\\
       b^{2}+W\left(b\right), &  W\left(b\right)\geq 0.
     \end{array}
   \right.
\end{displaymath}

\noindent
Thus, if the discriminant of $W$ is positive, there exists an interval $I$ on which $W$
is negative. The functions $g$ and $h_{1}$ coincide on this interval and for $b\in I$, $a=g\left(b\right)$ the condition \eqref{assumpab} is satisfied. $L$ is constant on this interval. It means that each $\left(g\left(b\right),b\right)$ for $b\in I$ corresponds to a local maximum of the likelihood function.
It is not difficult to show that $L$ does not attain any higher values anywhere else. Therefore there are infinitely many $\Psi\otimes\Gamma$ that maximize the likelihood function, so the maximum likelihood estimator for the covariance matrix is not uniquely defined.

If the discriminant of $W$ is negative, the equation $g(b)=h_{1}(b)$ is never satisfied
and solving the equation $g(b)=h_{2}(b)$ for $b$
leads to three solutions. Two of them involve the square root of the discriminant of $W$, thus they are not real. The third one is real and it corresponds to a $b$ such that $g\left(b\right)>b^{2}$.
This means that only in this case there is exactly one solution of the likelihood equations. It is easy to see that
this solution satisfies \eqref{assumpab} and corresponds to a unique global maximum.
So in this case the maximum likelihood estimator of the covariance matrix is uniquely defined.

We note that the discriminant of $W$ is a continuous function
of  continuous random variables on $(-\infty,+\infty)$.
Since for the two data sets corresponding to Figures~ \ref{plot_nonuniqueness} and \ref{plot_uniqueness} the discriminant of $W$ is positive and negative, respectively, the probability that the discriminant is positive and the probability that it is negative are both positive.
If the discriminant of $W$ equals zero, $g(b)=h_{1}(b)=h_{2}(b)=b^2$ in exactly one point $(a=b^2,b)$.
However, for this point the assumption \eqref{assumpab} is not satisfied,
and this solution is not allowed.  The event of the discriminant
of $W$ being zero has probability 0, though.

In conclusion, by computing the discriminant of $W$ for a given data set it can be determined whether there is unique maximum likelihood estimator for $\Psi\otimes\Gamma$ or not. Moreover, it is straightforward to simulate a data set for which maximum likelihood estimator of $\Psi\otimes\Gamma$ is not unique. Having such a data set and using the function $g$, one can obtain all possible values of $\Psi\otimes\Gamma$ that correspond to the global maxima of the likelihood function.

\section{Proof of Theorem \ref{matrixB}}
\label{uniqueness_diag}
\noindent
For proving Theorem \ref{matrixB} we need the following lemma.
\newtheorem{lem}[necessary]{Lemma}
\begin{lem}
\label{lem}
Suppose $\gamma_{i_{1}},\gamma_{i_{2}},\lambda_{i_{1}},\lambda_{i_{2}},
\psi_{j_{1}},\psi_{j_{2}},\phi_{j_{1}},\phi_{j_{2}}>0$.
If

\noindent\begin{tabularx}{\textwidth}{@{}XXX@{}}
  \begin{equation}
  \gamma_{i_{1}}\psi_{j_{1}}-\lambda_{i_{1}}\phi_{j_{1}}\leq 0,
\label{first_ineq}
  \end{equation} &
  \begin{equation}
  \gamma_{i_{1}}\psi_{j_{2}}-\lambda_{i_{1}}\phi_{j_{2}}\geq 0,
\label{second_ineq}
  \end{equation} \\
  \begin{equation}
  \gamma_{i_{2}}\psi_{j_{1}}-\lambda_{i_{2}}\phi_{j_{1}}\geq 0,
\label{third_ineq}
  \end{equation} &
  \begin{equation}
  \gamma_{i_{2}}\psi_{j_{2}}-\lambda_{i_{2}}\phi_{j_{2}}\leq 0,
\label{fourth_ineq}
  \end{equation}
\end{tabularx}
then it holds that
\begin{displaymath}
\gamma_{i_{1}}\psi_{j_{1}}-\lambda_{i_{1}}\phi_{j_{1}}=0,
\qquad
\gamma_{i_{1}}\psi_{j_{2}}-\lambda_{i_{1}}\phi_{j_{2}}=0,
\end{displaymath}
\begin{displaymath}
\gamma_{i_{2}}\psi_{j_{1}}-\lambda_{i_{2}}\phi_{j_{1}}=0,
\qquad
\gamma_{i_{2}}\psi_{j_{2}}-\lambda_{i_{2}}\phi_{j_{2}}=0.
\end{displaymath}
\end{lem}
\begin{proof}

\noindent
From \eqref{first_ineq} and \eqref{third_ineq} we obtain that $\frac{\lambda_{i_{2}}}{\gamma_{i_{2}}}\phi_{j_{1}}\leq\frac{\lambda_{i_{1}}}{\gamma_{i_{1}}}\phi_{j_{1}}$. Because $\phi_{j_{1}}>0$, we thus have $\frac{\lambda_{i_{2}}}{\gamma_{i_{2}}}\leq
 \frac{\lambda_{i_{1}}}{\gamma_{i_{1}}}$.
Similarly, from \eqref{second_ineq} and \eqref{fourth_ineq} we get $\frac{\lambda_{i_{2}}}{\gamma_{i_{2}}}\phi_{j_{2}}\geq\frac{\lambda_{i_{1}}}{\gamma_{i_{1}}}\phi_{j_{2}}$, and because $\phi_{j_{2}}>0$, we find $\frac{\lambda_{i_{2}}}{\gamma_{i_{2}}}\geq\frac{\lambda_{i_{1}}}{\gamma_{i_{1}}}$.
Therefore it must hold that $\frac{\lambda_{i_{2}}}{\gamma_{i_{2}}}=\frac{\lambda_{i_{1}}}{\gamma_{i_{1}}}$.
Using this result, we can rewrite inequalities
\eqref{third_ineq} and \eqref{fourth_ineq}:
\begin{equation}\label{third}
\psi_{j_{1}}\geq\frac{\lambda_{i_{1}}}{\gamma_{i_{1}}}\phi_{j_{1}},
\end{equation}
\begin{equation}\label{fourth}
\psi_{j_{2}}\leq\frac{\lambda_{i_{1}}}{\gamma_{i_{1}}}\phi_{j_{2}}.
\end{equation}

\noindent
Equations \eqref{first_ineq} and \eqref{third} imply $\psi_{j_{1}}=\frac{\lambda_{i_{1}}}{\gamma_{i_{1}}}\phi_{j_{1}}$, while \eqref{second_ineq} and \eqref{fourth} imply $\psi_{j_{2}}=\frac{\lambda_{i_{1}}}{\gamma_{i_{1}}}\phi_{j_{2}}$.
Therefore $\gamma_{i_{1}}\psi_{j_{1}}=\lambda_{i_{1}}\phi_{j_{1}}$
and $\gamma_{i_{1}}\psi_{j_{2}}=\lambda_{i_{1}}\phi_{j_{2}}$.
Because $\frac{\lambda_{i_{2}}}{\gamma_{i_{2}}}=\frac{\lambda_{i_{1}}}{\gamma_{i_{1}}}$,
it holds that $\psi_{j_{1}}=\frac{\lambda_{i_{2}}}{\gamma_{i_{2}}}\phi_{j_{1}}$ and
$\psi_{j_{2}}=\frac{\lambda_{i_{2}}}{\gamma_{i_{2}}}\phi_{j_{2}}$,
so that we also have  $\gamma_{i_{2}}\psi_{j_{1}}=\lambda_{i_{2}}\phi_{j_{1}}$
and $\gamma_{i_{2}}\psi_{j_{2}}=\lambda_{i_{2}}\phi_{j_{2}}$, and the lemma is proved.
\end{proof}

\noindent
We are now ready to prove Theorem \ref{matrixB}.

\begin{proof}[\textbf{Proof of Theorem \ref{matrixB}}\newline]
We will prove the theorem by induction.
First assume $p=2$, $q\geq 2$.
Suppose $\gamma_{1},\gamma_{2},\psi_{1},\ldots,\psi_{q}>0$ and $\lambda_{1},\lambda_{2},\phi_{1},\ldots,\phi_{q}>0$.
Consider the matrix
\begin{displaymath}
B=
\begin{bmatrix}
\left(\gamma_{1}\psi_{1}-\lambda_{1}\phi_{1}\right) & \ldots & \left(\gamma_{1}\psi_{q}-\lambda_{1}\phi_{q}\right)\\
\left(\gamma_{2}\psi_{1}-\lambda_{2}\phi_{1}\right) & \ldots & \left(\gamma_{2}\psi_{q}-\lambda_{2}\phi_{q}\right)
\end{bmatrix}
\end{displaymath}
and suppose that $B$ satisfies the assumptions of Theorem \ref{matrixB}.
It will be shown that all elements of $B$ are zero.

Take an arbitrary $k\in\{1,\ldots,q\}$ and suppose that  $B(1,k)$ is non-negative.
This implies that $B(2,k)$  must be non-positive.
Now consider the first row of $B$. There exists $r\in\{1,\ldots,q\}$ such that $r\neq k$ and
$B(1,r)$ is non-positive. Therefore $B(2,r)$ is non-negative.
This means that by Lemma \ref{lem} the elements $B(1,k)$, $B(2,k)$, $B(1,r)$, and
$B(2,r)$ of $B$ are zero.
If, instead, $B(1,k)$ is non-positive, analogous reasoning yields the same result.
Hence, we have shown that all elements in the $k$-th column of $B$ are equal to zero.
Since $k$ was arbitrary, all elements of  $B$ are zero.

Next, assume that for $2\leq p\leq k$, $q\geq 2$ it holds that if $B$ satisfies
the assumptions of Theorem \ref{matrixB},
then all elements of $B$ are equal to zero.
Consider the case $p=k+1$, $q\geq 2$.
Let $b_{k+1}$ denote the $\left(k+1\right)$-th row of $B$, and
 $B_{-(k+1)}$  the matrix $B$ with $b_{k+1}$ omitted.
We can have one of two situations:
\begin{enumerate}
\item $B_{-(k+1)}$ satisfies the assumptions of Theorem \ref{matrixB} for $\left(p,q\right)=\left(k,q\right)$.
\item $B_{-(k+1)}$ has a column that only contains non-zero elements of the same sign.
\end{enumerate}
In the first situation, we obtain from the inductive assumption for $\left(p,q\right)=\left(k,q\right)$
that all elements of  $B_{-(k+1)}$ equal zero. Thus, because $B$ satisfies the assumptions of Theorem \ref{matrixB}, all elements of $b_{k+1}$ are zero too.

In the second situation,  there exists at least one column of  $B_{-(k+1)}$ such that the
signs of all elements of this column are the same, say positive.
Therefore the first $k$ elements of the corresponding column of $B$ are
positive and the last one negative.
Because $B$ satisfies the assumptions of Theorem \ref{matrixB}, row $b_{k+1}$
must contain an element which is non-negative.
Let $s$ be the column of this element. Again due to
the assumptions of Theorem \ref{matrixB}, in the $s$-th column of $B$, in a row different from the $\left(k+1\right)$-th there is an element which is negative:

\begin{displaymath}
B=
\begin{bmatrix}
\ldots & + & \ldots & \ldots & \ldots \\
\ldots & + & \ldots & \ldots & \ldots \\
\vdots & & \vdots & \vdots \\
\ldots & \oplus & \ldots & \ominus & \ldots \\
\vdots & & \vdots & \vdots \\
\ldots & + & \ldots & \ldots & \ldots \\
\ldots & \ominus & \ldots & \oplus & \ldots \\
\end{bmatrix}.
\end{displaymath}

\noindent
By Lemma \ref{lem}, all circled elements are zero. We can repeat this reasoning to show that all positive elements of $b_{k+1}$ are zero. Since $B$ satisfies the assumptions of Theorem \ref{matrixB},
$b_{k+1}$ cannot contain any strictly negative elements either. It shows that
also in this situation all elements of this row are zero.
But this means that $B_{-(k+1)}$ must also satisfy the assumptions of Theorem \ref{matrixB}
and that by the inductive assumption  all elements of  $B_{-(k+1)}$ are zero too.
This finishes the proof of the theorem.
\end{proof}

\section*{Acknowledgements}
\noindent
This work was financially supported by a STAR-cluster grant from the Netherlands Organization of Scientific Research.

\end{document}